\documentclass[a4paper,11pt]{paper}
\usepackage[english]{babel}
\usepackage[utf8]{inputenc}
\usepackage{amsmath,amssymb,amsthm,mathrsfs,mathtools,graphicx,color,tikz}
\usepackage{hyperref}
\usepackage{caption}
\usepackage{parskip}

\renewcommand\>[1]{\vec{#1}}
\newtheoremstyle{sans}{\parskip}{\parskip}{\itshape}
                       {0pt}{\bfseries\sffamily}{.}{ }{}
\newtheoremstyle{sansplain}{\parskip}{\parskip}{}
                       {0pt}{\bfseries\sffamily}{.}{ }{}

\theoremstyle{sans}
\newtheorem{prop}{Proposition}[section]

\newtheorem{thm}[prop]{Theorem}
\newtheorem{lem}[prop]{Lemma}
\newtheorem{rem}[prop]{Remark}

\theoremstyle{sansplain}

\newtheorem*{theorem*}{Theorem}

\makeatletter
\@addtoreset{equation}{section}
\makeatother

\makeatletter
\@addtoreset{figure}{section}
\makeatother

\newcommand\C{\mathcal{C}}

\newcommand\Lc{\mathcal{L}}
\newcommand\Oc{\mathcal{O}}
\newcommand\N{\mathbb{N}}
\newcommand\Pb{\mathbb{P}}
\newcommand\Xf{\mathbf{X}}
\newcommand\Ta{\mathcal{T}}

\renewcommand{\geq}{\geqslant}
\renewcommand{\leq}{\leqslant}

\def\Qc{\mathcal{Q}}
\def\DD{\displaystyle}

\newcommand\1{\leavevmode\hbox{\rm \small1\kern-0.35em\normalsize1}}

\def\egaldef{\stackrel{\mbox{\tiny def}}{=}}
\renewcommand\>[1]{\vec{#1}}

\begin{document}

\title{Thermodynamical limits for  models of car-sharing systems: the Autolib' example}
\author{Guy Fayolle\thanks{INRIA  Paris, 48 rue Barrault, 75013 Paris, France. Email: {\tt Guy.Fayolle@inria.fr}}    \and
        Christine Fricker \thanks{INRIA Paris, 48 rue Barrault, 75013 Paris, France. E.mail: {\tt Christine.Fricker@inria.fr}} \vspace*{0.7cm}\\ \hspace*{2.5cm} \emph{January 2025}}
 \date{}
\maketitle
\begin{abstract}
We analyze mean-field equations obtained  for models  motivated by a large station-based car-sharing system in France called \emph{Autolib’}. The main focus is on a version where  users reserve a parking space  when they take a car. In a first model, the reservation of parking spaces is effective for all users (see \cite{bourdais2020mean}) and capacity constraints are ignored.  The model is carried out in thermodynamical limit, that is  when the number $N$ of stations and the number of cars $M_N$ tend to infinity, with $U=\lim_{N \rightarrow \infty}M_N/N$. This limit is described by Kolmogorov's equations of a two-dimensional time-inhomogeneous Markov process depicting the numbers of reservations  and cars at a station. It satisfies a non-linear differential system. We prove analytically that this system has a unique solution, which converges, as $t\to\infty$, to an equilibrium point exponentially fast. Moreover, this equilibrium point corresponds to the stationary distribution of a two queue tandem (reservations, cars),  which is here always ergodic. The intensity factor of each queue has an explicit form obtained from an intrinsic mass conservation relationship. 
Two related models with capacity constraints are briefly presented in the last section: the simplest one with no reservation leads to a one-dimensional problem; the second one corresponds to our first model with finite total capacity~$K$.
\end{abstract}
.

\section{Introduction}\label{sec:intro}
The paper investigates asymptotic properties for some time-inhomogeneous Markov processes  obtained as mean-field limits. More precisely, in a system with $N$ sites and $M_N$ particles, they describe the limiting behavior of an arbitrary site in the system,  when the numbers $N$ and $M_N$  get large together at the same speed, i.e. when $M_N/N$ tends to a constant $U>0$ as $N\to\infty$. The limit thus obtained is generally called  the {\em thermodynamic limit} in statistical physics. 

Such a  Markov process is described by its transient probability given by  the forward  Kolmogorov's equations. We prove the existence and uniqueness of a solution to these non-linear differential equations, as well as an exponential rate of convergence  toward the stationary regime. This is a main issue, which appears to be crucial in the literature to obtain further results on the mean-field convergence as, for example, the so-called {\em refined mean-field} in \cite{gast2017refined, allmeier2022mean}.

This question has also been investigated for finite state reversible  dynamics \cite{budhiraja2015local,budhiraja2015limits}, involving the construction of Lyapounov functions by using relative entropy. Our study goes beyond this framework with a simple model  emanating from car-sharing systems (CSS) (see \cite{bourdais2020mean}), which can also be represented as systems of interacting particles. Here,  stations are viewed as sites and cars become particles. The main feature of these station-based CSS (e.g. \emph{Autolib'}) is that users can reserve a parking slot at their destination. 

\subsection{A particle  model with reservation}\label{sec:particle} 
Let us describe the  model as an homogeneous version of simplified car-sharing dynamics. Consider $M_N$ particles randomly spread among $N$ sites, each site having a capacity $K\leq +\infty$, so that $M_N\le KN$ when $K$ is finite. Particles can be of two types, either a reserved parking space (reservation) or a car. At time $t=0$, the repartition of the particles  among the sites is arbitrary.
From  site $i$, after a random time exponentially distributed with parameter $\lambda$, a car decides to go to site $j$ with probability $1/N$ and a reservation appears at site~$j$.  If the reservation is possible, i.e. when the total number of particles at site $j$ is less than $K$, then the car leaves $i$, and the  reservation at $j$ remains for a travel time exponentially distributed with parameter $\mu$ (the case $i=j$ is possible, with the same parameters). Otherwise, no car departure occurs. 

It should be noted that another dynamics would be to allow the particle of type car to  randomly search for an unsaturated site, but we shall not consider this model.

\subsection{A queueing description of the model}\label{sec:queueing}  
We will recast the mean field system introduced above in a queueing context, and it will be useful to recall the classical notation, which can be found e.g. in~\cite{GnKo}.

$M/M/1/\infty$ stands for a single server queue, with Poisson input, exponentially distributed service times, infinite capacity. In this study, the service discipline can be FIFO (first-in-first-out), LIFO (last-in-first-out) or RAND (random order). When the Poisson intensity is time-dependent, we shall write $M(t)/M/1/\infty$; otherwise, we shall write
$./M/1/\infty$ for an unspecified arrival point process.

Similarly, $M/M/\infty$ (resp. $M(t)/M/\infty$) denotes a queue with infinitely many servers, Poisson input with constant (resp. time dependent) intensity and exponentially distributed service times.

In the queueing systems just defined, arrivals and service times are independent.

The  model proposed in Section~\ref{sec:particle} can alternatively be viewed as a closed Jackson network (see \cite{BCMP}), with $M_N$ customers. This means that each of the  $N$ sites can be replaced by   a   tandem with total capacity $K$.  The customers in the first infinite-server queue of  tandem  $i$ represent reservations, and customers in the second $M/M/1/\infty$ queue represent available cars.   The service times of the tandem are, respectively, the travel times of the cars and the inter-arrival times of the users. Both are independent and exponentially distributed, with respective parameters  $\mu$ and  
$\lambda$. Note that the routing considered here involves a blocking policy. Indeed, if the   destination  site (tandem) is saturated, the user request is rejected.

\subsection{Main outcomes}
\begin{itemize}
    \item Section~\ref{sec:mod1} is devoted to a mean-field model of the type presented in Section~\ref{sec:particle}, in which reservation of parking spaces is effective for all users (see \cite{bourdais2020mean}) and capacity constraints are ignored ($K=\infty$). We prove existence and uniqueness of the equations describing the evolution of an underlying dynamical system. The related stationary distribution has a product form corresponding to a Jackson network of two queues in tandem, where the intensity factor in each queue  is obtained as the unique admissible root of a second degree equation, derived from a mass conservation relationship. In addition, the speed of convergence toward equilibrium is exponential.
    \item Section~\ref{sec:thm} contents the proof of the main theorem stated in Section~\ref{sec:mod1}.
    \item In Section~\ref{sec:others}, two related models with capacity constraints are presented: one without reservation, which leads to a single queue problem, the second being merely the model of Section~\ref{sec:mod1} with finite $K$.
    \item The Appendix provides some technical properties, in particular stochastic dominance, concerning the dependence of time-inhomogeneous birth and death processes on the parameters.
\end{itemize}

\section{Thermodynamical limit in the unbounded capacity case}\label{sec:mod1}
In this section and in Section~\ref{sec:thm}, we focus on the case $K=\infty$, referring to Section~\ref{sec:others} for variants of the model with finite~$K$.

Although  system dynamics  may seem simpler when $K=\infty$, its study is still far from complete. For related topics, we refer for instance to \cite{DeFa, fricker2017equivalence}.

For any fixed $N$, let us denote by $R^N_i(t)$ and $V^N_i(t)$ the number of reserved slots and cars at site $i \;(1\leq i\leq N)$, at time $t\ge0$.
Then the  process 
\[
   (R^N(t),V^N(t))=\{(R^N_i(t),V^N_i(t)),1\leq i\leq N\} 
\]
is Markovian and  irreducible  on the finite state-space
\begin{align*}
 \mathcal{S}^N=\{(r,v)=(r_i,v_i)_{ 1\le i \le N}\in \N^{2N}, \sum_{i=1}^N (r_i+v_i) =M_N \} , 
\end{align*}
with a unique  invariant probability measure. Its infinitesimal generator is given by 
\begin{align*}
 G^Nf(r,v)&= \frac{\lambda}{N}\sum_{i,j=1}^N \big( f(r+e_j, v-e_i)-f(r,v)\big)1_{v_i>0} \\
 &+ \mu  \sum_{i=1}^{N} r_i \big( f(r-e_i,v+e_i)-f(r,v)\big), \quad \forall (r,v)\in \mathcal{S}^N,
\end{align*}
for any function $f(\cdot,\cdot)$ with support on $\mathcal{S}^N$.
The so-called  \emph{mean-field} approach is used to investigate the large-scale behavior of the system. We start from the empirical measure given, for any pair 
$(j,k)\in\N^2 $ and $t\ge0$, by 
\begin{align*}
   \alpha^N_{j,k}(t)=\frac{1}{N}\sum_{i=1}^N 1_{\{(R^N_i(t),V^N_i(t))=(j,k)\}} .
\end{align*}
With this definition, $\alpha^N_{j,k}(t)$ represents the proportion of tandems (sites)  at time $t$, with $j$ and $k$ customers in their respective queues. 

 Let $\>\alpha^N(t) \egaldef \{\alpha^N_{j,k}(t)\}_{(j,k) \in \N^2}$. Note that, since there are $M_N$ particles in the system as defined in Section~\ref{sec:particle}, we have the following conservation equality
\begin{align*}
  \sum_{j=1}^N (j+k)\alpha^N_{j,k}(t)=\frac{M_N}{N}, \quad t\ge0. 
\end{align*}

In the sequel, we are mainly interested in the limit behavior of the system as $N\to\infty$, assuming 
\[
U\egaldef\DD\lim_{N\to\infty}\frac{M_N}{N}.
\]
Letting $\Lambda$ denote the set of probability measures $\{\rho_{j,k}, j,k\geq 0\}$ on $\N^2$, subject to the constraint 
\begin{equation}\label{eq:Lambda}
\sum_{j,k\geq 0}  (j+k)\rho_{j,k}=U,
\end{equation}
the following proposition holds.

\begin{prop}
  Assume $\>\alpha^N(0)$ converges weakly, as $N\to\infty$, to some fixed distribution 
 $\>\alpha_0\in \Lambda$. Then  the empirical measure  
$\>\alpha^N(t)$ converges  in distribution to a deterministic dynamical system denoted by $\>\alpha(t)$, which  satisfies  the  
following infinite system of nonlinear forward Kolmogorov's equations
\begin{align} \label{eq:kolmo1}
 \frac{d\alpha_{j,k}(t)}{dt}+\left[\lambda b(t)+\mu j +\lambda \1_{\{k>0\}}\right]\alpha_{j,k}(t) & = \nonumber \\ 
 \lambda b(t)\alpha_{j-1,k}(t)\1_{\{j>0\}} 
+\mu(j+1)\alpha_{j+1,k-1}(t)\1_{\{k>0\}} &+\lambda \alpha_{j,k+1}(t), \quad (j,k) \in \N^2, \ t\ge0,
\end{align}
 where $\>\alpha(t)\in \Lambda$, $\>\alpha(0)= \>\alpha_0$ and $\DD b(t) \egaldef 1-\sum_{j\in\N} \alpha_{j,0}(t)$.
\end{prop}
\begin{proof}
This result was demonstrated in~\cite{fricker2022mean}. It could also be proved along the lines proposed in~\cite{DeFa}, which heavily rely on standard theoretical tools involving the convergence of generators (see~\cite{ethier2009markov}).

\end{proof}
\begin{rem}
    It will be shown in Section~\ref{sec:additional} (see~\eqref{eq:mass2bis}), that the mass conservation equation
    \[ 
     \sum_{j,k\geq 0}  (j+k)\alpha_{j,k}(t)=U, \quad \forall t\ge0,
\]
 is in fact intrinsic to the system~\eqref{eq:kolmo1}.
\end{rem}

\subsection{Mean-field equations viewed as a tandem of two queues} 

System \ref{eq:kolmo1} represents the joint distribution of the number of units in the tandem shown in Fig.~\ref{fig:TypicalObject}. The first queue stands for the reservations and is of  
$M(t)/M/\infty$ type, with arrival (resp. service) rate $\lambda b(t)$ (resp. $\mu$), while  the second queue containing the cars is of $./M/1/\infty$ type, with FIFO discipline and service rate $\lambda$. Here $b(t)$ is the probability that there is at least one car in the second queue.

\begin{figure}[!ht]
\begin{tikzpicture}[scale=0.8]



\node (middle1) at (2.8,0.3) {$\lambda\, b(t)$};
\draw (middle1);
\draw[->] (1.5,0) -- (4.2,0);

\draw[-] (5.5,-1.5) -- (5.5,1.5);

\foreach \y in {-1.5,-1.2,...,1.5}
\draw[-] (4.5,\y) -- (5.5,\y);

\draw[-] (4.5,1.5) -- (5.5,1.5);

\node (R) at (5,2) {$R$};
\draw (R);
\node (roR) at (5,-2) { $j$};
\draw (roR);


\node (middle2) at (5.8,1.5) {$\mu$};
\draw (middle2);
\draw[->] (5.8,0) -- (7.2,0);

\draw[-] (7.5,.5) -- (9.5,.5);
\draw[-] (7.5,-.5)-- (9.5,-.5);
\draw[-] (9.5,.5) -- (9.5, -.5);
\draw[-] (9,.5) -- (9, -.5);
\draw[-] (8.5,.5) -- (8.5, -.5);

\node (V) at (8.4,1.4) {$V$};
\draw (V);
\node (ro2) at (8.5,-1) {$k$};
\draw (ro2);


\node (middle3) at (9.5,0.9) {$\lambda$};
\draw (middle3);
\draw[->] (9.6,0) -- (11,0);


\end{tikzpicture}
\caption{Dynamics of the tandem. The first queue is of $M(t)/M/\infty$ type and the second queue is a simple $./M/1/\infty$ queue. }
\label{fig:TypicalObject}
  \end{figure}
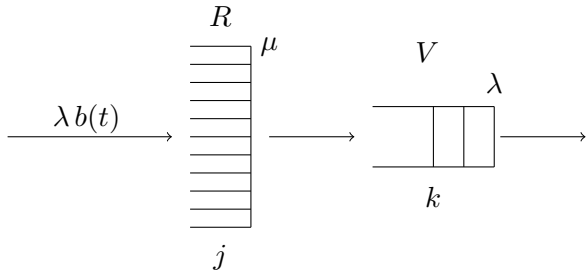
 Introduce the generating function
\[
F(x,y,t) \egaldef \sum_{j,k\ge0} \alpha_{jk}(t) x^j y^k,
\]
where $x,y$ are complex variables such that $0\le |x|,|y| \le1$. Then, \eqref{eq:kolmo1} is equivalent to the nonlinear functional equation
\begin{equation}\label{eq:S2}
\frac{\partial F(x,y,t)}{\partial t} + \left[\lambda b(t)(1-x)+\lambda \Bigl(1-\frac{1}{y}\Bigr)\right]
F(x,y,t) =
\mu(y-x)\frac{\partial F(x,y,t)}{\partial x} +\lambda \Bigl(1-\frac{1}{y}\Bigr) F(x,0,t),
\end{equation}
where $b(t) = 1- F(1,0,t)$.

\subsection{Results}\label{sec:main}
Equation~\eqref{eq:kolmo1} can be viewed as the limit of mean-field equations subject to the capacity constraint $j+k\le K$ as $K\to\infty$. In this case, we will show that the solution of~\eqref{eq:kolmo1} always exist, and that the stationary regime is reached exponentially fast, for any $\lambda,\mu>0$. The global situation is pictured in the next theorem.
\begin{thm}\label{thm:S1}
\begin{enumerate}\mbox{ } 
\item For each $t\ge0$ there exists a unique distribution $\>\alpha(t)$ satisfying the nonlinear system~\eqref{eq:kolmo1}. In addition, for any $\lambda,\mu>0$, as $t\to\infty$, the distribution
$\>\alpha(t)$ converges in ~$\Lambda$ (defined by~\eqref{eq:Lambda})  to the unique equilibrium point $\>\pi >\>0$ given by 
    \begin{equation}\label{eq:invar2}
        \pi_{j,k}=e^{-\rho\beta} \frac{(\rho\beta)^j}{j!}(1-\beta)\beta^k, \quad \forall j,k\in \N,
    \end{equation}
    where
    \begin{equation}\label{eq:param}
     \rho\egaldef\frac{\lambda}{\mu}, \ \mathrm{and} \ 
     \beta\egaldef \lim_{t\to\infty}b(t) = 
     \frac{U+\rho+1-\sqrt{(U+\rho+1)^2-4\rho U}}{2\rho} < 1,
    \end{equation}
  $\DD U= \lim_{N\to\infty}M_N/N$  being the constant introduced in Section~\ref{sec:intro}.
\item  There is a mass conservation of the form 
\begin{equation}\label{eq:mass2}
U=\sum_{j,k\geq 0} (j+k)\alpha_{j,k}(t), \quad 0\le t\le\infty,
    \end{equation}
  which, although somehow surprising, follows implicitly from the  functional equation~\eqref{eq:S2}. 

\item Assuming at $t=0$ all particles are reserved parking spaces, i.e. 
$\>\alpha(0)$ has its support in $\N\times \{0\}$, then, as $t\to\infty$,
\begin{equation}\label{eq:vt1}
    b(t) = \beta + \Oc(e^{-vt}), 
   \end{equation}
 where
 \begin{equation}\label{eq:vt2}
 v = \min[ \mu, \lambda\bigl(\sqrt{1}-\sqrt{\beta}\bigr)^2],
 \end{equation}
 and $1/v$ being viewed as the relaxation time of the system.
 \end{enumerate}
\end{thm}

This theorem shows in particular that the stationary distribution of \eqref{eq:S2} is given by~\eqref{eq:stat}, which has a product-form solution. This corresponds to the  network shown in Figure~\ref{fig:TypicalObject}, for $t=\infty$, with constant rates, whose joint-distribution is well-known to have a product-form (see~\cite{BCMP}).

\section{Proof of Theorem \ref{thm:S1}}\label{sec:thm}
 The line of argument (both of analytic and probabilist nature) consists in viewing the second queue of our tandem as a $M(t)/M/1/\infty$ queue with service rate~$\lambda$. Its arrival rate $\delta(t)$ defined in~\eqref{eq:delta} is theoretically computable via system~\ref{eq:sys}, and directly related to $b(t)$ introduced  in~\eqref{eq:S2}. In a second step, we show via an iterative scheme the convergence of 
$\delta(t)$ to a finite constant $\delta$, whose explicit form is given in Lemma~\ref{lem:pform}. The speed of convergence is obtained in Section~\ref{sec:rate}.

\subsection{Subsidiary equations}\label{sec:additional}
Beforehand, we establish three useful functional equations  which follow immediately from~\eqref{eq:S2}, keeping in mind that $x,y$ are complex variables with $|x|,|y| \leq 1$.
 
 \subsubsection{Equation for \texorpdfstring{$F(x,x,t)$}{}} 
Setting $G(x,t)\egaldef F(x,x,t)$, with $G(1,t)=1$, equation~\eqref{eq:S2} yields
\begin{align}\label{eq:G}
 \frac{\partial G(x,t)}{\partial t}+\lambda \left(b(t)-\frac{1}{x} \right) (1-x)  G(x,t)=\lambda \left( 1-\frac{1}{x} \right)F(x,0,t).
\end{align}
Then, dividing equation~\eqref{eq:G} by $1-x$ and letting~$x\to1$, we get 
\[
\frac{\partial }{\partial t} \frac{\partial }{\partial x}G(1,t) =0,
\]
or, equivalently,
\begin{equation}\label{eq:mass2bis}
      \frac{\partial G(1,t)}{\partial x} = U,
   \end{equation} 
thus proving the mass conservation~\eqref{eq:mass2}, stated in Section~\ref{sec:main}, and the point \emph{(ii)} of the theorem. 

\begin{rem}\label{rem:E2} It might be interesting to note that the conclusion~\eqref{eq:mass2bis} still holds if, in equation~\eqref{eq:S2}, $b(t)$ would be replaced by any \emph{meaningful} function $b(x,y,t)$ satisfying $b(1,1,t) =1-F(1,0,t)$.
\end{rem}
\subsubsection{Equation for \texorpdfstring{$F(1,y,t)$}{}}
Setting $A(y,t)\egaldef F(1,y,t)$, with $A(1,t)=1$ and  $A(0,t)=1-b(t)$, equation~\eqref{eq:S2} yields
\begin{align}\label{eq:A}
 \frac{\partial A(y,t)}{\partial t}+\lambda \left(1-\frac{1}{y} \right) A(y,t) + \mu(1-y)\frac{\partial F(1,y,t)}{\partial x} = 
 \lambda \left( 1-\frac{1}{y} \right)A(0,t).
\end{align}
We shall refer to this equation in Section~\ref{sec:deltat}.
\subsubsection{Eqution for \texorpdfstring{$F(x,1,t)$}{}}
 Setting $B(x,t)\egaldef F(x,1,t)$, with $B(1,t)=1$, and $\DD r(t)\egaldef\frac{\partial F(1,1,t)}{\partial x}$, equation~\eqref{eq:S2} yields immediately
\begin{align}\label{eq:B}
 \frac{\partial B(x,t)}{\partial t}+\lambda b(t)(1-x) B(x,t)=\mu(1-x)\frac{\partial B(x,t)}{\partial x}.
\end{align}
Dividing now equation~\eqref{eq:B}~by $1-x$ and letting $x\to1$, we obtain
\begin{align}\label{eq:S2r}
\frac{d r(t)}{dt}= \lambda b(t)-\mu r(t),
\end{align}
noting that $r(t)$ represents the mean number of customers in the first queue of the tandem at time~$t$. 

\begin{rem}\label{rem:E3} Equation \eqref{eq:S2r} seems quite natural, since it can be viewed as a fluid equation for the first queue of the tandem, which is a $M(t)/M/\infty$  queue with arrival rate $\lambda b(t)$ and service rate $\mu$. It can also be obtained directly from equation~\eqref{eq:S2}, by taking the partial derivative with respect to $y$, then substituting $x=y=1$ and using \eqref{eq:mass2bis}.
\end{rem}
The first queue is a $M(t)/M/\infty$ queue with arrival rate $\lambda b(t)$ and service rate $\mu$. 
\paragraph{Assumption H1}
 The first queue is supposed to be empty at time zero.
 
 Then the departure process of this queue is a non-homogeneous Poisson process, see e.g.~\cite[p.99]{GnKo}, with intensity 
$$\delta(t)\egaldef \int_0^t \lambda b(t-x) \mu e^{-\mu x}dx, \ t>0,$$
or, after solving~\eqref{eq:S2r} with the condition $r(0)=0$ 
(which corresponds to \textbf{H1}),
\begin{align}\label{eq:delta}
    \delta(t)=\lambda\mu \int_0^t b(s) e^{-\mu(t-s)} ds=\mu r(t), \ t>0.
\end{align}
For the sake of clarity, let us emphasize that in the rest of Section~\ref{sec:thm} the variable $r(t)$ will not appear explicitly, and we will focus on $\delta(t)$.

\subsection{A differential system for \texorpdfstring{$\delta(t)$}{}} \label{sec:deltat}
We suppose that at time $t=0$, the first queue 
$M(t)/M/\infty$ is empty. This state can be reached with probability one by a direct coupling argument, due to  the underlying Markovian evolution given by~\eqref{eq:kolmo1}. 
Then, from the previous section, it appears that the second queue of the tandem, denoted by $\Ta_2$, is of $M(t)/M/1$ type, with constant service rate~$\lambda$ and Poisson arrival intensity 
$\delta(t)$, defined by~\eqref{eq:delta}. 

Let $N(t)$ denote the random number of units (namely the number of cars) in~$\Ta_2$ at time~$t$, and its associated generating function
\[
\DD N(y,t)\egaldef\sum_{k\ge0}\Pb(N(t)=k) y^k.
\]
The original system thus behaves in the same stochastic way as the two following queues in interaction:
\begin{itemize}
    \item $\Ta_2$, just defined above;
    \item $\Ta_1$, which is of $M(t)/M/\infty$ type, which Poisson arrival rate $\lambda(1-N(0,t))$ and service rate~$\mu$.
\end{itemize}
We analyze $\Ta_2$ in more detail, starting from 
the forward Kolmogorov's equations, which give rise to the functional equation
\begin{align}\label{eq:A1}
 \frac{\partial N(y,t)}{\partial t}+\left(\lambda \left(1-\frac{1}{y} \right) +\delta(t)(1-y)\right) N(y,t) =
 \lambda \left(1-\frac{1}{y}\right)N(0,t),
\end{align}
where $\delta(t)$ is given by~\eqref{eq:delta}, keeping in mind the similarity between equations~\eqref{eq:A} and \eqref{eq:A1}.

The direct integration of \eqref{eq:A1}  leads to
\begin{align}\label{eq:Asol}
    \varphi(y,t)N(y,t)=N(y,0)+\lambda \left(1-\frac{1}{y}\right)\int_0^t \varphi(y,s)N(0,s)ds,
\end{align}
where 
\begin{align}
    \varphi(y,t)=\exp \left\{ \left(1-\frac{1}{y}\right) \int_0^t \bigl(\lambda - y\delta(s)\bigr) ds \right\}.
\end{align}
Setting, for notational convenience, $H(t)\egaldef \DD N(0,t)= 1-b(t)$, we have  from Cauchy's formula
$\DD H(t)=\frac{1}{2i\pi}\int_{\Lc}\frac{N(z,t)}{z}dz$, where $\Lc$ stands for a simple closed contour around~$0$.

Then, we get from~\eqref{eq:Asol} the following Volterra integral equation of the second kind
\begin{align}\label{eq:H}
    H(t)=\psi(t)+\lambda \int_0^t D(s,t)H(s) ds,
\end{align}
with 
\begin{equation*}
    \psi(t)=\frac{1}{2i\pi}\int_{\cal{L}}\dfrac{N(z,0)}{z\varphi(z,t)}dz, \qquad D(s,t)=\frac{1}{2i\pi}\int_{\cal{L}}\dfrac{\varphi(z,s)}{\varphi(z,t)} \dfrac{1-1/z}{z}dz.
\end{equation*}
It turns out that $D(s,t)$ can be expressed in terms of the modified Bessel functions
\[
I_n(z) = \frac{1}{2i\pi}\int_{|\omega|=1} \frac{\exp[(z/2)(\omega+1/\omega)]}{\omega^{n+1}} d\omega = 
\frac{1}{\pi}\int_0^\pi e^{z\cos(\theta)}\cos(n\theta)
d\omega, \quad n\in \mathbb{Z}.
\]
Setting $\DD A_t \egaldef \int_0^t \delta(s)ds$, we have exactly
\begin{equation}
    D(s,t) = e^{-\lambda(t-s)-(A_t-A_s)}
    \biggl(I_0\bigl(2\sqrt{\lambda(t-s)(A_t-A_s)}\bigr) -
\biggl(\frac{A_t-A_s}{\lambda(t-s)}\biggr)^{1/2}I_1\bigl(2\sqrt{\lambda(t-s)(A_t-A_s)}\bigr)\biggr).
\end{equation}
Hence, upon combining equations~\eqref{eq:S2r}, \eqref{eq:delta} and~\eqref{eq:H}, we have proved the following result.
\begin{prop}
The rate function $\delta(t)$ satisfies the differential system  
\begin{equation}\label{eq:sys}
\begin{cases}
    \DD \frac{d \delta(t)}{dt} + \mu \delta(t)& = \lambda\mu(1-H(t)), \quad \delta(0)=0, 
    \\[0.3cm]
     H(t) & =  \psi_\delta(t)+\lambda \int_0^t D(s,t)H(s) ds.
\end{cases}
\end{equation}
\end{prop}
\begin{rem}
  The first equation of system~\eqref{eq:sys} shows that
$\DD
\frac{d \delta(t)}{dt} + \mu \delta(t) \leq \lambda\mu$, with $\delta(0)=0$, whence 
\[
\delta(t) \leq \lambda(1-e^{-\mu t}).
\]
\end{rem} 

In the sequel, and in particular in the proof of the forthcoming lemma, to emphasize the functional dependence on $\delta(t)$, we will add the subscript $\delta$ to the concerned functions,~e.g. $H_\delta, \psi_\delta,D_\delta,$~etc. This convention also applies to \eqref{eq:sys}.  
\begin{lem}\label{lem:dlimit}
The integro-differential system~\eqref{eq:sys} has a unique solution 
$\delta(\cdot)$ with
\begin{equation}\label{eq:dlimit}
    \delta(0)= 0, \quad\mathrm{and} \quad \delta \egaldef \lim_{t\to\infty}\delta(t)
\end{equation}
exists and is a finite positive constant.
\end{lem}
\begin{proof}
According to a method similar to that proposed in~\cite{DeFa}, we introduce the following iteration scheme:
\begin{equation}\label{eq:scheme1}
  \begin{cases}
      \delta_0(t)=0, \quad \forall t\ge 0, \\[0.2cm]
   \DD \frac{d \delta_{n+1}(t)}{dt} + \mu \delta_{n+1}(t) =
   \lambda\mu (1-H_{\delta_n}(t)),  \\[0.3cm]
   \delta_{n+1}(0) = 0, \quad \forall n\ge 0, \\[0.2cm]
   \DD H_{\delta_n}(t) =  \psi_{\delta_n}(t)+\lambda \int_0^t D_{\delta_n}(s,t)H_{\delta_n}(s) ds .
  \end{cases} 
  \end{equation}
We will show that this scheme is increasing.
For each $n\ge0$, let $\mathcal{Q}_n$ denote the $M(t)/M/1/\infty$ queue with FIFO service discipline, arrival rate $\delta_n(t)$ and service intensity 
$\lambda$. Then, the probability for  $\mathcal{Q}_n$ to be empty at time~$t$ is equal to $H_{\delta_n}(t)$, the unique solution of the Volterra integral equation appearing in system~\eqref{eq:scheme1}.
 Using the simple stochastic monotonicity proved in Appendix~\ref{sec:A1}, we shall argue by induction
on $n$, assuming that all $\mathcal{Q}_n$'s have the same initial
conditions [this argument is rendered possible by the third equation in \eqref{eq:scheme1}].

Suppose $\delta_n(t)\ge \delta_{n-1}(t)$, which is in particular true for~$n=1$. Then, by property~(ii) of Lemma~\ref{lem:A1}, we have the inequality
$H_{\delta_n}(t)\le H_{\delta_{n-1}}(t)$, which, after rewriting the second equation of~\eqref{eq:scheme1} as
\begin{equation}\label{eq:dnt}
\delta_{n+1}(t) = \lambda\mu \int_0^t e^{\mu(s-t)}(1-H_{\delta_n}(s))ds,
\end{equation}
yields immediately $\delta_{n+1}(t)\ge\delta_n(t)$. Hence, the uniformly bounded sequence $\{\delta_n(t),n\geq 0 \}$ (resp. $\{H_{\delta_n}(t),n\geq 0\}$) is, for each fixed~$t$, non-decreasing (resp. non-increasing), and the functions
\begin{alignat*}{2}
\delta(t) & = \lim_{n\rightarrow\infty}\delta_n(t)  \quad \text{and}
 \quad H_{\delta}(t) &  = \lim_{n\rightarrow\infty}H_{\delta_n}(t) 
\end{alignat*}
 do exist and satisfy~\eqref{eq:sys}.

After having shown the existence of a solution to system~\eqref{eq:sys}, we
are left with the problem of uniqueness. In fact, considering the
first equation of \eqref{eq:sys} and using standard results on integral
equations, this uniqueness is straightforward, since by~(iii) of
Lemma~\ref{lem:A1}, $H_{\delta}(\cdot)$ satisfies a Lipschitz condition
with respect to $\delta(\cdot)$. 

\begin{lem} \label{lem:pform} Any \emph{reachable} stationary distribution
of the tandem driven by equation~\eqref{eq:S2} has a product form $F(x,y)\DD \egaldef\lim_{t\to\infty} F(x,y,t)$ in the sense that 
\begin{equation}\label{eq:pform}
F(x,y) = e^{p(x-1)}\,\frac{1-q}{1-qy},
\end{equation}
\end{lem}
with $p>0$, $0\le q<1$ and $\mu p = \lambda q$.
\begin{proof}
  Letting $t\to\infty$ in \eqref{eq:S2}, we obtain the following stationary Kolmogorov's equations of a \emph{quasi-reversible} system, which, after setting 
  $F(x,y)\egaldef F(x,y,\infty)$, have the form
\begin{equation}\label{eq:stat}
\left(\delta (1-x) + \lambda \biggl(1-\frac{1}{y}\biggr)\right) F(x,y) = \mu(y-x)\frac{\partial F(x,y)}{\partial x} + 
\lambda \biggl(1-\frac{1}{y}\biggr)F(x,0).
\end{equation}
  Then, inserting \eqref{eq:pform} into~\eqref{eq:stat},
we obtain immediately the necessary relations
\begin{equation}\label{eq:stat1}
\delta = \mu p = \lambda q = \lambda(1-F(1,0)),
\end{equation}
and the proof of Lemma~\ref{lem:pform} is concluded.
\end{proof}
To proceed further with the proof of Lemma~\ref{lem:dlimit}, we use the mass conservation equation~\eqref{eq:mass2bis}, which says that the total mean number of units in the system is equal to some given bounded constant~$U$, for all~$t\ge0$. Hence, any possible reachable stationary regime corresponds necessarily to a positive recurrent process. Although we do not yet know the behavior of $\delta(t)$ when $t\to\infty$, we see that, for any positive increasing sequence $\{t_k,k\ge0\}$, with $\DD\lim_{k\to\infty}t_k=\infty$ and 
$\DD\lim_{k\to\infty}\delta(t_k)=\overline{\delta}$, the equilibrium  equation~\eqref{eq:stat} is obtained as the limit of~\eqref{eq:S2}, where $\delta$ is replaced by $\overline{\delta}$. In addition, 
$\overline{\delta}$ must satisfy~\eqref{eq:stat1}, together with the relation
\[
U= \frac{\partial F(1,1)}{\partial x} + \frac{\partial F(1,1)}{\partial y} = p + \frac{q}{1-q},  
\]
or, equivalently,
\[
U =\frac{\overline{\delta}}{\mu} + \frac{\overline{\delta}}{\lambda-\overline{\delta}},
\]
which yields a second degree equation for $\overline{\delta}$ having the unique admissible root
\begin{equation}\label{eq:deltabar}
\overline{\delta} = \frac{\mu\bigl(U+\rho+1-\sqrt{(U+\rho+1)^2-4\rho U}\bigr)}{2}.
    \end{equation}
Therefore $\overline{\delta}$ does not depend on the choice of the sequence~$\{t_k\}$, hence the existence of the limit
\[\delta = \lim_{t\to\infty}\delta{(t)} = \overline{\delta},
\]
 given by \eqref{eq:deltabar}.
The proof of Lemma~\ref{lem:dlimit} is concluded.
\end{proof}

The preceding arguments show that we are entitled to analyze the equations \eqref{eq:kolmo1}, \eqref{eq:S2}, \emph{treating 
$\lambda b(t)$ as if it were the exogenous
function (see \eqref{eq:S2r}, \eqref{eq:delta})
\[
\lambda b(t)= \delta(t) + \frac{1}{\mu}\frac{d\delta(t)}{dt},
\]
where $\delta(t)$ satisfies~\eqref{eq:sys}}.
After doing this, equations~\eqref{eq:kolmo1} ---describing the evolution of the process~$\xi(t)$ associated with the original  tandem queue network shown in Figure~\ref{fig:TypicalObject})--- become standard forward Kolmogorov's equations, whose existence and uniqueness of solution follow at once from the general theory (see e.g. \cite{FEL,ReLe}), and they satisfy the condition
\[
\sum_{j,k}\alpha_{j,k}(t)=1, \quad \forall t<\infty.
\]
To study the steady-state behavior of~$\xi(t)$ when $t\to\infty$, we start by claiming, on the basis of equations~\eqref{eq:mass2} and \eqref{eq:mass2bis}, that $\xi(t)$ is necessarily ergodic, which implies in particular~$\delta<\lambda$. 

Moreover, the invariant measure of $\xi(t)$ is expected to coincide with  the invariant measure of the random walk, 
say $\widetilde{\xi}(t)$, obtained just replacing the function 
$\delta(t)$ by the constant~$\delta$. This statement can be established by a simple coupling argument, relying on the continuity of~$\delta(t)$. Indeed, for all~$\varepsilon>0$, there exists $T_{\varepsilon}$ such that
\begin{equation}\label{eq:eps}
|\delta(t)-\delta | < \epsilon, \quad \forall t\geq T_{\epsilon}.
\end{equation}
Taking now $\widetilde{\xi}(0) = \xi(T_{\epsilon})$, the result
follows directly from inequality \eqref{eq:A5} and Section~\ref{sec:A1}.

 On the other hand, $\widetilde{\xi}(t)$ corresponds to
a standard Jackson network known to have the product-form stationary distribution (see Lemma~\ref{lem:pform})
\[
 \pi_{j,k}=e^{-(\delta/\mu)}\,\frac{(\delta/\mu)^j}{j!}
 (1-\beta)\beta^k, \quad \forall j,k\in \N,  
\]
and one checks directly  from~\eqref{eq:deltabar} that $\DD \beta=\frac{\delta}{\lambda}$ is  given by~\eqref{eq:param}. 

This concludes the proof of the point \emph{(i)} of Theorem~\ref{thm:S1}.


\subsection{Speed of convergence of \texorpdfstring{$\delta(t)$}{}}\label{sec:rate}
The following proposition provides an estimate of the speed of convergence of $\delta(t)$ toward~$\delta$. 
\begin{prop}\label{prop:speed}\mbox{}
As $t\to\infty$, $\delta(t) = \delta + \Oc(e^{-vt})$, with $v=\min\bigl(\mu,\bigl(\sqrt{\lambda} - \sqrt{\delta}\bigr)^2\bigr)$.
\end{prop}
\begin{proof}
Let $\widetilde{\Qc}_n, n\ge0,$ be the following sequence of ergodic
$M/M/1/\infty$ queues with the following parameters for each $\widetilde{\Qc}_n$:
\begin{itemize}
\item the service rate is equal to 
$\lambda$; 
\item the intensity of the Poisson arrival process has the constant value 
$\gamma_n<\lambda$.
\end{itemize}
Let $\widetilde{H}_n(t)$ be the probability for
$\mathcal{\widetilde{Q}}_n$ to be empty at time $t$. It is well-known that $\widetilde{H}_n(t)$ can be expressed in terms of Bessel integrals (see e.g. \cite[p.~23]{Takacs}). Moreover, setting
\[c_n\egaldef\bigl(\sqrt{\lambda} - \sqrt{\gamma_n}\bigr)^2,
\]
we have, as $t\rightarrow\infty$, the
classical estimate, valid here since $\lambda>\alpha$ (see
e.g. \cite[p.~107]{ASMU}),
\begin{equation}\label{eq:asmu}
\widetilde{H}_n(t) = 1- \dfrac{\gamma_n}{\lambda} + 
\widetilde{h}_nt^{-3/2} e^{-c_nt} + o\bigl(t^{-3/2}e^{-c_nt}\bigr),
\end{equation}
 where $\widetilde{h}_n$ is a uniformly bounded constant depending on
the initial conditions. 

\begin{lem}\label{lem:HS}
 Let $f(\cdot)$ be an integrable function of $t$ on 
 $\mathbb{R}_+$, such that 
 $\DD l=\lim_{t\to\infty}f(t)$ exists. Then 
\begin{equation}\label{eq:conv}
\lim_{t\to\infty}\int_0^t \mu e^{\mu(s-t)} f(s)ds = l.
\end{equation} 
\end{lem}
\begin{proof}
 Just splitting the integral into the sum  
 $\DD\int_0^T +\int_T^t$, where, for a given $\varepsilon$, $T$ is such that
 \[ 
 |l-f(t)|< \varepsilon ,\quad \forall t\ge T,
 \]
the proof of the lemma is immediate by letting $t\to\infty$, and then $\varepsilon\to 0$.
\end{proof}

Let us consider now the following scheme
\begin{equation}\label{eq:scheme2}
  \begin{cases}
      \gamma_0(t)= \gamma \in]0, \lambda[,\quad \forall t\ge 0, 
      \\[0.2cm]
   \DD \frac{d \gamma_{n+1}(t)}{dt} + \mu \gamma_{n+1}(t) = 
   \lambda\mu(1-H_{\gamma_n}(t)),  \\[0.3cm]
   \gamma_{n+1}(0) = 0, \quad \forall n\ge 0,\\[0.2cm]
   \DD H_{\gamma_n}(t) =  \psi_{\gamma_n}(t)+\lambda \int_0^t D_{\gamma_n}(s,t)H_{\gamma_n}(s) ds,
  \end{cases} 
  \end{equation}
which differs from \eqref{eq:scheme1} only by its first equation, but this is not a minor point! 

The last equation of \eqref{eq:scheme2} can be seen as emanating from a $M(t)/M/1$ ergodic queue, which we suppose to be empty at time~$0$. Then, by Takács formula (see~\cite[Th.1, p. 23]{Takacs}), we can write 
\[
\widetilde{H}_n(t)\ge 1-\frac{\gamma_n}{\lambda}, \quad \forall  \gamma_n\le\lambda.
\]
Using now~Lemma~\ref{lem:HS} and~\eqref{eq:dnt} with $\delta_n$ replaced by $\gamma_n$, we  get at once $\gamma_1(t)\le\gamma(1-e^{-\mu t})\le \gamma$, which implies by induction that $\gamma_n(t)$ is monotone decreasing in $n$, and 
\[
\lim_{n\rightarrow\infty}\gamma_n(t) =\delta(t), \quad \mathrm{with} \quad
\delta_n(t)\leq\delta(t)\leq\gamma_n(t), \quad \forall t,n\geq 0.
\] 
Suppose, and this is true for $n=0$,
\[
\gamma_n(t) = \gamma_n + \Oc(e^{-a_n t}).
\]
Using the monotonicity properties already mentioned and the estimate~\eqref{eq:asmu}, we can write
\[
H_{\gamma_n}(t) = 1- \dfrac{\gamma_n}{\lambda} + 
\Oc\bigl(\max(e^{-a_nt},t^{-3/2}e^{-c_n t})\bigr).
\]
Then, the integrated form 
\begin{equation}\label{eq:gt}
\gamma_{n+1}(t) = \lambda\mu \int_0^t e^{\mu(s-t)}(1-H_{\gamma_n}(s))ds
\end{equation}
 yields (some details are omitted),   
\begin{align*}
   \gamma_{n+1}(t) = \gamma_n + \Oc(e^{-a_{n+1}t})
\end{align*}
where 
\[
a_{n+1} = \min{(a_n,v_n)}, \quad \mathrm{with} \ v_n = \min{(\mu,c_n)}.
\]
Choosing now $\gamma_0(t)=\delta, \forall t\ge 0$, in the first equation of~\eqref{eq:scheme2}, we have 
\[
\lim_{n\to\infty} a_n= \lim_{n\to\infty}\min(\mu, (\sqrt\lambda -\sqrt\gamma_n)^2) = (\sqrt\lambda -\sqrt\delta)^2,
\]
concluding the proof of Proposition~\ref{prop:speed}. 
\end{proof}
 
 Applying now Lemma~\ref{lem:HS} to the integral form obtained  from~\eqref{eq:sys},
 \[
\delta(t) = \lambda\mu \int_0^t e^{\mu(s-t)}(1-H_{\delta}(s))ds,
 \] 
 we get $\delta=\lambda \beta$.
 Hence, Proposition~\ref{prop:speed} is clearly equivalent to point \emph{(iii)} of  Theorem~\ref{thm:S1}, the proof of which is complete.
\hfill  $\blacksquare$

\section{Two other related systems} \label{sec:others}
In this section, we present two systems in the same context with capacity constraints. Both  can be  analyzed via the same methods, and the main results are stated without detailed proofs.

\subsection{Model 2: mean-field limit viewed as a single  \texorpdfstring{$M(t)/M/1/K$}{}  queue with finite capacity \texorpdfstring{$K$}{}}
This model is the simplest one among the car-sharing class  (see e.g.~\cite{fricker2017equivalence}), with no parking space reservation. It is particularly meaningful for bike-sharing systems, which do not permit reservation.
There are $M_N$ cars moving among $N$ stations of capacity $K$. A car leaves the station at rate $\lambda$ and travels for a random time  (exponentially distributed with  parameter 
$\mu$). Then, it returns to a uniformly chosen station, if possible, or else sets off on a new journey.

As $N\to \infty$ with $\DD U=\lim_{N\to \infty}M_N/N$, the usual empirical distribution tends to a deterministic dynamical system $\>\alpha(t)$ described by  the following equations, for $0\leq j\leq K,  \,t\ge0$,
\begin{equation}\label{eq:kolmo2}
  \frac{d\alpha_j(t)}{dt}+\left[\lambda\1_{\{j>0\}} + 
  \mu (U-a(t))\1_{\{j<K\}}\right]\alpha_j(t) = 
  \mu (U-a(t))\alpha_{j-1}(t) +\lambda \alpha_{j+1}(t) \1_{\{j<K\}},
\end{equation}
where  $\DD a(t)= \sum_{j=0}^K j\alpha_j(t)$.
The underlying time-inhomogeneous Markov process is a 
$M(t)/M/1/K$ queue with service rate $\lambda$ and Poisson arrival rate $\mu(U-a(t))$, where $a(t)$ is the mean number of customers in the queue at time~$t$.

Setting 
$\DD Q(z,t) = \sum_{j=0}^K \alpha_j(t) z^j$, for an arbitrary  
 complex variable $z$ and $t\ge0$, we get the nonlinear functional equation
\begin{equation}\label{eq:S1}
\frac{\partial Q(z,t)}{\partial t} + \left[\mu (U-a(t))(1-z)+\lambda\Bigl(1-\frac{1}{z}\Bigr)\right]Q(z,t) = \lambda\Bigl(1-\frac{1}{z}\Bigr) Q(0,t)
+\mu (U-a(t))\alpha_K(t)z^K(1-z),
\end{equation}
with $\DD a(t)= \frac{\partial Q(1,t)}{\partial z}$.

\begin{thm}\label{thm:S2}
For $t\ge0$, there exists a unique distribution $\>\alpha(t)$ satisfying the nonlinear system~\eqref{eq:kolmo2}. In addition, for any $\lambda,\mu>0$, there is a unique equilibrium point
$\DD\lim_{t\to\infty}\>\alpha(t)=\>\pi >\>0$ given by 
\[ 
\pi_j = \frac{(1-\beta)\beta^j}{1-\beta^{K+1}}, \quad 0\le j\le K,
\]
where 
\[
\beta=\lim_{t\to\infty} \frac{\mu}{\lambda}\bigl(U-a(t)\bigr) = 
\frac{\mu}{\lambda}\Bigl(U-\sum_{j=0}^K j\pi_j\Bigr)
\]
is the unique solution of the fixed point equation
\[
\beta = \frac{\mu}{\lambda}\biggl(U- \frac{\beta}{1-\beta}+
\frac{(K+1)\beta^{K+1}}{1-\beta^{K+1}}\biggr).
\]
\end{thm}
\begin{proof}
System~\eqref{eq:kolmo2}  rewrites in the matrix form
\begin{equation}\label{eq:mat2}
  \frac{d\>\alpha(t)}{dt} = \>\alpha(t)M(a(t)),
 \end{equation}
  where $\>\alpha(t)$ is a row vector and $M(a(t))$ a tridiagonal matrix,  denoted in this way to emphasize the dependence on $a(t)$. Now we introduce the  iterative vector scheme
\begin{equation}\label{eq:scheme3}
  \begin{cases}
      a_0(t) =  \>\alpha_0(t) =0,  \quad t\ge 0, \\[0.1cm]
     \DD a_n(t) = \sum_{j=0}^K j\alpha_{n,j}(t), \\[0.4cm]
     \>\alpha_{n+1}(0) =\>\alpha(0), \\[0.2cm]
    \DD \frac{d\>\alpha_{n+1}(t)}{dt} = \>\alpha_{n+1}(t)M(a_n(t)), \quad n\ge0,  
  \end{cases} 
  \end{equation}
 which can be shown, using Appendix~\ref{sec:A1}, to be increasing in the sense that $a_n(t)<a_{n+1}(t)<U$.

\end{proof}

\subsection{Model 3: mean-field limit viewed as a system of two queues in tandem with total capacity bounded by \texorpdfstring{$K<\infty$}{}}
The model is   the same as that described in Sections~\ref{sec:particle} and~\ref{sec:queueing}, with a finite capacity constraint $K$. The first (resp. second) queue of the tandem is still an $M(t)/M/\infty$ 
(resp. $./M(t)/1$) queue, but the total number of particles remains bounded by~$K$. Then, for $j,k\in \N$, $0\leq k+j\leq K$, the following mean-field equations hold (see~\cite{bourdais2020mean}), for all~$t\ge0$,
\begin{align}\label{eq:kolmo3}
  \frac{d\alpha_{j,k}(t)}{dt}+\left[\lambda d(t)\1_{\{j+k<K\}}+\mu j 
  +\lambda c(t)\1_{\{k>0\}}\right]\alpha_{j,k}(t)&= \nonumber\\
 \lambda d(t)\alpha_{j-1,k}(t) +\mu(j+1)\alpha_{j+1,k-1}(t)
  & +\lambda c(t)\alpha_{j,k+1}(t)\1_{\{j+k<K\}},
\end{align}
where 
\begin{equation} \label{eq:ratesS3}  
d(t)= 1-\sum_{j=0}^K \alpha_{j,0}(t),\quad  \mathrm{and} \quad 
c(t)= 1-\sum_{j=0}^K \alpha_{j,K-j}(t).
\end{equation}
Setting
\[
F(x,y,t) \egaldef \sum_{0\leq j+k\leq K} \alpha_{jk}(t) x^j y^k,
\quad \forall t\ge0,
\]
for $x,y$ arbitrary complex variables, the following functional equation holds.
\begin{align}\label{eq:S3}
&\frac{\partial F(x,y,t)}{\partial t} + \left[\lambda d(t)(1-x)+\lambda c(t)\Bigl(1-\frac{1}{y}\Bigr)\right]
F(x,y,t) \nonumber\\
& \quad = \mu(y-x)\frac{\partial F(x,y,t)}{\partial x} +\lambda c(t)\Bigl(1-\frac{1}{y}\Bigr) F(x,0,t)+\lambda d(t)(1-x) F_K(x,y,t),
\end{align}
where $F_K(x,y,t)=\sum_{j=0}^K \alpha_{j,K-j}(t)x^j y^{K-j}$, $d(t) = 1- F(1,0,t)$ and $c(t) = 1- F_K(1,1,t)$.

The analog of Theorem~\ref{thm:S1} can be stated as follows.
\begin{thm}\label{thm:S3}
For each $t\ge0$ there exists a unique distribution $\>\alpha(t)$ satisfying the nonlinear system~\eqref{eq:kolmo2}. In addition, for any $\lambda,\mu>0$, there is a unique equilibrium point
$\DD\lim_{t\to\infty}\>\alpha(t)=\>\pi >\>0$ given by 
\begin{align} \label{eq:invar3}
\pi_{jk} = \frac{1}{Z}\frac{\rho_R^j}{j!}{\rho_V^{k}}, \quad 0\le j+k\le K,
\end{align}
where $Z$ is a normalizing constant and $(\rho_R,\rho_V)$ is obtained as the unique solution of the  system 
\begin{equation}\label{eq:toto}
\begin{cases}
\rho_R &=\DD\frac{\lambda}{\mu}\biggl( 1-\DD\sum_{j=0}^K\pi_{j0}\biggr),\\[0.4cm]
U&= \DD\sum_{j,k} (j+k)\pi_{jk}.
\end{cases}
\end{equation}
\end{thm}

As previously announced, we only sketch out a few arguments. To this end, as in Section~\ref{sec:additional}, we write for Model~3 the following subsidiary equations, which emanate directly from~\eqref{eq:S3}.

\paragraph{(1)} Letting $G(x,t)\egaldef F(x,x,t)$, with $G(1,t)=1$,  we have
\begin{align}\label{eq:G3}
 \frac{\partial G(x,t)}{\partial t}+\lambda \left(d(t)-\frac{c(t)}{x} \right) (1-x)  G(x,t)=\lambda c(t)\left( 1-\frac{1}{x} \right)F(x,0,t)+\lambda d(t) (1-x) x^K(t)(1-c(t)).
\end{align}
 Then, dividing \eqref{eq:G3} by $1-x$ and letting~$x\to1$, we get 
\[
\frac{\partial }{\partial t} \frac{\partial }{\partial x}G(1,t) =0,
\]
or, equivalently,
\begin{equation}\label{eq:mass2bis3}
      \frac{\partial G(1,t)}{\partial x} = U,
   \end{equation} 
which depicts a mass conservation valid for $t\ge0$.

\paragraph{(2)} For $A(y,t)\egaldef F(1,y,t)$, with $A(1,t)=1$ and  $A(0,t)=1-d(t)$, equation~\eqref{eq:A} still holds, just replacing $\lambda$ by $\lambda c(t)$.  
\begin{align}\label{eq:A3}
 \frac{\partial A(y,t)}{\partial t}+\lambda c(t)\left(1-\frac{1}{y} \right) A(y,t) = \mu(1-y)\frac{\partial F(1,y,t)}{\partial x} + 
 \lambda c(t)\left( 1-\frac{1}{y} \right)A(0,t).
\end{align}

\paragraph{(3)} Analogously, letting $B(x,t)\egaldef F(x,1,t)$, with $B(1,t)=1$, 
\begin{align*}
 \frac{\partial B(x,t)}{\partial t}+\lambda  d(t)(1-x) B(x,t)=\mu(1-x)\frac{\partial B(x,t)}{\partial x}+\lambda  d(t)(1-x) F_K(x,1,t),
\end{align*}
which yields the equivalent of~\eqref{eq:S2r}, just replacing 
$\lambda$ by $\lambda c(t)$, i.e.
\begin{align}\label{eq:S2r3}
\frac{d r(t)}{dt}= \lambda c(t)d(t)-\mu r(t),
\end{align}
where $\DD r(t)\egaldef\frac{\partial F(1,1,t)}{\partial x}$ stands for the mean number of customers in the first queue of the tandem, at time~$t$. 

\begin{rem}\label{rem:F3} As equation~\eqref{eq:S2r}, equation \eqref{eq:S2r3} seems quite natural. Indeed, it can be viewed as a fluid equation for the first queue of the tandem, which is a $M(t)/M/\infty$  queue with arrival rate $\lambda c(t)d(t)$ and service rate $\mu$, due to the additional capacity constraint, which translates into a  probability of acceptance $d(t)$.
\end{rem}
The following result is the analog of Lemma~\ref{lem:pform}.
\begin{lem} \label{lem:pform3} Any \emph{reachable} stationary distribution
of the tandem driven by equation~\eqref{eq:S3} has the generating function $F(x,y)\DD \egaldef\lim_{t\to\infty} F(x,y,t)$, 
\begin{equation}\label{eq:pform3}
F(x,y) = \dfrac{1}{Z}\sum_{j+k\leq K} \frac{(px)^j}{j!}\,(qy)^k,
\end{equation}
where $\DD Z=\sum_{j+k\leq K} \frac{p^j}{j!}\,q^k$ is the normalizing constant and where  $p>0$ and $0\le q<1$ are uniquely determined by $\mu p = \lambda  d$ with 
$\DD d= 1-Z^{-1}\sum_{j=0}^K p^j/j!$,
and by the mass conservation equation 
\begin{align}\label{eq:mass3}
    U=\dfrac{1}{Z}\sum_{j+k\leq K} (j+k)\frac{p^j q^k}{j!}.
\end{align}
\end{lem}

\begin{proof}
Letting $t\to\infty$ in equation~\eqref{eq:S3}, we obtain the following stationary Kolmogorov's equations of a \emph{quasi-reversible} system, which, after setting 
  $F(x,y)\egaldef F(x,y,\infty)$, have the form
\begin{align}\label{eq:stat3}
&\left(\lambda d(1-x) + \lambda c\biggl(1-\frac{1}{y}\biggr)\right) F(x,y) \nonumber\\
&= \mu(y-x) \frac{\partial F(x,y)}{\partial x} + 
\lambda c \biggl(1-\frac{1}{y}\biggr)F(x,0)+\lambda d(1-x)F_K(x,y),
\end{align}
where $c=c(\infty)$ and $d=d(\infty)$.

  Then, inserting \eqref{eq:pform3} into~\eqref{eq:stat3}, we obtain  (after some easy algebra) the necessary relation 
\begin{equation*}
  \mu p = \lambda d,
\end{equation*}
which, combined with the mass conservation~\eqref{eq:mass3}, yields a system of two equations  with two unknowns $p$ and $q$, having a unique solution (see~\cite[Theorem 3]{fricker2022mean} for details). The proof of Lemma~\ref{lem:pform3} is concluded.
\end{proof}

From the previous lemma, we get directly \eqref{eq:invar3}, which corresponds to the last point of Theorem~\ref{thm:S3}.
\appendix 
\section{Queues with time-dependent arrival rates}
\subsection{The \texorpdfstring{$M(t)/M/1/\infty$}{} queue} \label{sec:A1}
This section briefly presents some basic features of the operator (birth and death process type) describing the evolution of the 
$M(t)/M/1/\infty$ queue, with time varying  arrival rate $\beta(\cdot)\in\C^+[0,\infty]$ and  constant service rate $\mu$.

 The probabilities $p_n(t)\egaldef \Pb(Z(t)=n), \,n\geq 0,$
where Z(t) denotes the number in the queue at time $t$, obey the
following set of {\em forward} differential equations:
\begin{equation}\label{eqA1}
\begin{cases}
\DD \dfrac{dp_0(t)}{dt}=-\beta(t)p_{0}(t)+\mu p_1(t),\\[0.3cm] \dfrac{dp_n(t)}{dt}=\beta(t)
p_{n-1}(t)-(\beta(t)+\mu)p_n(t)+\mu p_{n+1}(t),\quad n\geq 1,
\end{cases}
\end{equation}
which we rewrite in operator form
\begin{equation}\label{eq:A2}
\dfrac{d\mathbf{P_{\beta}}(t)}{dt} =
\mathbf{P_{\beta}}(t)\mathbf{K_{\beta}}(t),
\end{equation}
 where $\mathbf{K_{\beta}}(t)$ is a generator (represented by an
infinite matrix) and $\mathbf{P_{\beta}}(t)$ is an infinite row vector
belonging to the Banach space $\ell_1$ of absolutely summable
sequences. It is known either from a probabilistic point of view
(e.g. \cite{FEL, ReLe}) or by an analytic argument (e.g. \cite{CARTAN}),
that (\ref{eq:A2}) has, for all $t\geq 0$, a unique solution in
$\ell_1$. In addition, the generator $\mathbf{K_{\beta}}(t)$ has a
continuous spectrum of eigenvalues, located on the negative real line.

Similarly, the distribution function
$$s_n(t) \egaldef \Pb(Z(t)\leq n), \quad \forall n\geq 0,$$ satisfies
the system
\begin{equation}\label{eqA3}
\begin{cases}
\dfrac{ds_0(t)}{dt}=-(\beta (t)+ \mu)s_{0}(t)+ \mu s_1(t),\\[0.3cm] \dfrac{ds_n(t)}{dt}=\beta(t) s_{n-1}(t)-(\beta(t)+\mu)s_n(t)+ \mu s_{n+1}(t),\quad n\geq 1,
\end{cases}
\end{equation}
which will be written as
\begin{equation}\label{eqA4}
\dfrac{d\mathbf{S_{\beta}}(t)}{dt} =
\mathbf{S_{\beta}}(t)\mathbf{L_{\beta}}(t), 
\end{equation}
where $\mathbf{S_{\beta}}(t)$ denotes the row vector
$\mathbf{S_{\beta}}(t) \egaldef (s_0(t),s_1(t),\ldots ).$ 

\begin{lem}\label{lem:A1} 
Let $\mathbf{P_{\beta}}(t)$ and $\mathbf{\widetilde{P}_{\beta}}(t)$ be the
solutions of (\ref{eq:A2}) corresponding to respective initial
conditions $\mathbf{P_{\beta}}(0)$ and
$\mathbf{\widetilde{P}_{\beta}}(0)$. The following properties hold:
\begin{itemize}
\item[{\rm (i)}] If $\mathbf{P_{\beta}}(0) \geq
\mathbf{\widetilde{P}_{\beta}}(0) \ge 0$, then $\mathbf{P_{\beta}}(t)
\geq \mathbf{\widetilde{P}_{\beta}}(t) \ge0, \ \forall t \geq 0.$
\item[{\rm (ii)}] Let $\beta(\cdot), \gamma(\cdot) \in \C^+[0,\infty]$, such
that 
\[\beta(t) \leq \gamma(t), \quad \forall t \geq 0, \quad \text{and}
 \quad \mathbf{S_{\beta}}(0) \geq \mathbf{S_{\gamma}}(0).
 \] 
 Then \emph{(stochastic dominance)}
 \[\mathbf{S_{\beta}}(t) \ge \mathbf{S_{\gamma}}(t) ,\ \forall t \ge0.
\]
In particular, with the notation of equation \eqref{eq:sys},
$H_{\beta}(t) \ge H_{\gamma}(t),\ \forall t\ge0$.
\item[\rm{(iii)}] For any $\beta(\cdot), \gamma(\cdot) \in \C^+[0,\infty]$,
with $\mathbf{P_{\beta}}(0)=\mathbf{P_{\gamma}}(0)$, we have \emph{the
Lipschitz condition} 
\begin{equation}\label{eq:A5}
|\mathbf{P_{\beta}}(t) - \mathbf{P_{\gamma}}(t)| \leq
K\|{\beta}(\cdot)-{\gamma}(\cdot)\|, 
\end{equation} where $K$ is an
absolute constant independent of $t$ and $|\cdot|$ denotes, for an
arbitrary vector $\Xf$, the usual $\ell_1$-norm $|\Xf|=\sum_i|x_i|$.
\item[{\rm (iv)}] The $M(t)/M/1/\infty$ queue with arrival
rate $\delta(t)$ and the $M/M/1/\infty$ queue with arrival
rate $\delta=\lim_{t\rightarrow\infty}\delta(t)$ have the same
stationary regime.
\end{itemize}
\end{lem}
\begin{proof}
As for point (i), it is not difficult to see that
$\mathbf{K_{\beta}}(t)$ and $\mathbf{L_{\beta}}(t)$ are positive
operators. For instance, making in \ref{eqA3} the change of functions
\[
s_n(t)= w_n(t)\exp\Bigl(-\mu t-\int_0^t\beta(s)ds\Bigr),\quad n\geq 0,
\]
leads to the system
\[\begin{cases}
\dfrac{dw_0(t)}{dt}= \mu w_{1}(t), \\[0.3cm]
\dfrac{dw_n(t)}{dt}=\beta(t)w_{n-1}(t) + \mu w_{n+1}(t),\quad
n\geq 1,
\end{cases}
\]
which has the form 
\[
\dfrac{d\mathbf{W_{\beta}}(t)}{dt} =
\mathbf{W_{\beta}}(t)\mathbf{\widetilde{L}_{\beta}}(t), 
\]
where $\mathbf{\widetilde{L}_{\beta}}(t)$ has only positive
coefficients.  A similar argument can be used for the positivity of
$\mathbf{K_{\beta}}(t)$.

\medskip
The stochastic dominance in (ii) follows now from (i). Indeed, setting
\[
\mathbf{R}(t) = \mathbf{S_{\beta}}(t)- \mathbf{S_{\gamma}}(t),
\]
the row vector $\mathbf{R}(t)$ satisfies the non homogeneous differential
equation 
\begin{equation}\label{eq:A6}
\dfrac{d\mathbf{R}(t)}{dt} = \mathbf{R}(t)\mathbf{L_{\gamma}}(t) +
\bigl(\gamma (t)-\beta(t)\bigr) \mathbf{D}(t),
\end{equation}
where $\mathbf{D}(t) = (d_0(t), d_1(t), \ldots )$, with $d_n(t) =
[s_{n+1}(t) -s_{n}(t)]_{\beta}\,.$ By (i), the vector $\mathbf{D}(t)$
has non-negative components and the operator $\mathbf{L_{\gamma}}(t)$
is positive, whence it follows that the solutions of \eqref{eq:A6} are
also non negative. 

To prove (iii), we shall use differential calculus in
 Banach spaces. In this framework, most of the classical results for
 the real line or the complex plane apply without substantial
 modification.

For any $\beta(\cdot)\in\C^+[0,\infty]$, with $0\leq\|\beta\|\leq B$, take an arbitrary perturbation function $\Delta(\cdot)$, with
$\beta(t) + \Delta(t) \in \C^+[0,\infty]$. When it exists, the
partial derivative with respect to $\beta(\cdot)$ of a differentiable mapping
$$g : \C^+[0,\infty]\times [0,\infty] \rightarrow \ell_1 $$ is a
 functional (see \cite{CARTAN}) written $\DD\frac{\partial g(t)}{\partial
 \beta}$. With this notation, one sees easily that $\DD\mathbf{Q}_{\beta}
 \egaldef \frac{\partial \mathbf{P_{\beta}}(t)}{\partial \beta}$,
 where $P_{\beta}$ satisfies (\ref{eq:A2}), must be a solution of the
 following non-homogeneous linear differential equation
\begin{equation}\label{eq:A7}
\dfrac{d\mathbf{Q}_{\beta}(t)}{dt} =
\mathbf{Q}_{\beta}(t)\mathbf{K_{\beta}}(t)+\mathbf{P_{\beta}}(t)\mathbf{M},
 \quad \mathbf{Q}(0) = 0,
\end{equation}
where $\mathbf{M}$ is a constant infinite matrix given by
\[
\begin{pmatrix}
-1 & 1 & 0 & 0 & \dots \\ 
0 & -1 & 1 & 0 & \dots \\ 
0 & 0 & -1 & 1 & \dots \\
0 & 0 & 0 & -1 & \dots \\
\vdots & \vdots & \vdots & \vdots & \dots
\end{pmatrix}.
\]
Then the solution of \eqref{eq:A7} writes in the form 
\[\mathbf{Q}_{\beta}(t) = \int_0^t \mathbf{P_{\beta}}(s)\mathbf{M}{\bf
\Phi}(t,s) ds, 
\]
where ${\bf \Phi}(t,s)$ is the so-called fundamental
solution (see \cite{CARTAN}) of the homogeneous equation of type
\eqref{eq:A2}. The preceding argument yields directly the rough estimate
\begin{equation}\label{eq:A8}
\sup_{\beta}\|\mathbf{Q}_{\beta}\| = \sup_{\beta}\sup_{t\geq
0}\|\mathbf{Q}_{\beta}(t)\|\leq K ,
\end{equation}
where $K$ is a bounded constant.  Since $\mathbf{Q}_{\beta}$ is the
derivative with respect to $\beta(\cdot)$ of the function
$\mathbf{P_{\beta}}$, defined on the Banach space $\C^+[0,\infty]$,
\eqref{eq:A8}) gives the  Lipschitz condition \eqref{eq:A5}.

The last property (iv) in the lemma can be viewed as an immediate consequence of the stochastic ordering contained in (i) and (ii), and details will be omitted.
The proof of the Lemma is concluded.

\end{proof}
\bibliography{FF-Autolib}
\bibliographystyle{acm}

\end{document}